\documentclass[11pt]{amsart}
\usepackage{latexsym}
\usepackage{amsfonts}
\usepackage{amsmath,amssymb}
\usepackage{color}

\newcommand{\field}[1]{\mathbb{#1}}
\newcommand{\C}{\field{C}}

\newcommand{\ignore}[1]{}

\newtheoremstyle{s2}{9pt}{9pt}{\rm}{}{\bf}{.}{0.5em}{}
\theoremstyle{s2}
\newtheorem{defi}{Definition}[section]
\newtheorem{ex}[defi]{Example}
\newtheorem{re}[defi]{Remark}

\newtheoremstyle{s1}{9pt}{9pt}{\it}{}{\bf}{.}{0.5em}{}
\theoremstyle{s1}
\newtheorem{lem}[defi]{Lemma}
\newtheorem{theo}[defi]{Theorem}
\newtheorem{co}[defi]{Corollary}

\DeclareMathOperator{\Sing}{Sing}
\DeclareMathOperator{\mult}{mult}

\DeclareMathOperator{\ord}{ord}

\DeclareMathOperator{\res}{res}
\DeclareMathOperator{\im}{Im}

\title[On stable polynomial mappings]{On stable polynomial mappings} \makeatletter

\@addtoreset{equation}{section}
\makeatother
\author{M. \ Farnik \& Z. Jelonek }
\address[M. Farnik]{Jagiellonian University\\
Faculty of Mathematics and Computer Science\\
{\L}ojasie\-wi\-cza~6, 30-348 Krak\'ow, Poland}
\email{michal.farnik@gmail.com}

\address[Z. Jelonek]{Instytut Matematyczny\\
Polska Akademia Nauk\\
\'Sniadeckich 8, 00-656 Warszawa, Poland}
\email{najelone@cyf-kr.edu.pl}


\subjclass{14 R 99, 32 A 10}
\thanks{The  authors are partially supported by the grant of Narodowe Centrum Nauki number 2019/33/B/ST1/00755}

\begin{document}

\begin{abstract}
For given natural numbers $d_1,d_2$ let $\Omega_2(d_1,d_2)$ be the set off all polynomial mappings $F=(f,g):\C^2\to\C^2$ such that $\deg f\le d_1$, $\deg g\le d_2$. We say that the mapping $F$ is  topologically stable in $\Omega_2(d_1,d_2)$ if for every small deformation $F_t\in \Omega_2(d_1,d_2)$ the mapping $F_t$ is topologically equivalent to the mapping $F$. The aim of this paper is to characterize the topologically stable mappings in $\Omega_2(d_1,d_2)$.
In particular we show how to effectively determine a member of $\Omega_2(d_1,d_2)$ with generic topology. \end{abstract}

\maketitle

\bibliographystyle{alpha}

\section{Introduction}
Polynomial mappings $F\colon \C^n\to\C^n$ are the most classical objects in complex analysis, yet the study of their topology has begun only recently.  The first paper in this direction was \cite{fjr}, where authors show that a ``Whitney type Theorem" is true for the family $\Omega_2(d_1,d_2)$, with the natural topology of $\Omega_2(d_1,d_2)$ (instead of the Whitney topology, which is discrete on $\Omega_2(d_1,d_2)$). 

Let us recall that $\Omega_2(d_1,d_2)$ is the set off all polynomial mappings $F=(f,g):\C^2\to\C^2$ such that $\deg f\le d_1$, $\deg g\le d_2$.  The set $\Omega_2(d_1,d_2)$ can be identified with $\C^N$, for the appropriate $N$, by identifying a pair of polynomials with the $N$-tuple of their coefficients. Thus, we have the Euclidean and the Zariski topologies on $\Omega_2(d_1,d_2)$. When we refer to map-germs, small deformations and small neighborhoods we use the Euclidean topology. When we refer to open subsets of $\Omega_2(d_1,d_2)$ we use the Zariski topology. Obviously a Zariski open set is an Euclidean open set, yet often we can show that an Euclidean open set is Zariski open by noticing that it is Zariski constructible.

In the paper \cite{fj} we classify all topological types of quadratic mappings $F: \C^2\to \C^2$. If we consider polynomial mappings of higher algebraic degree then the problem of topological classification becomes very difficult.

However, in  papers \cite{jel}, \cite{fjr} it is shown that there exists a Zariski open, dense subset $U$ of $\Omega_2(d_1,d_2)$ such that every two mappings $f,g\in U$ are topologically equivalent and locally topologically stable (in the classical sense). If a mapping $G\in \Omega_2(d_1,d_2)$ is topologically equivalent to some mapping $F\in U$ we say that $G$ has a \emph{generic topological type}. In this paper we would like to characterize the mappings with generic topological type. We show that a mapping $F\in \Omega_2(d_1,d_2)$ has generic topological type if and only if it is \emph{topologically stable}, i.e., for every small deformation $F_t\in \Omega_2(d_1,d_2)$ of $F$, the mapping $F_t$ is topologically equivalent to the mapping $F$. In particular we show that mappings with generic topological type form a Zariski open subset in $\Omega_2(d_1,d_2)$. 
 
Our notion of topological stability, which is natural in the context of polynomial mappings, is different from the classical notion of topological stability of smooth mappings. Note that in our case we can not use the Whitney topology (it is discrete on $\Omega_2(d_1,d_2)$). In fact there are polynomial mappings which are stable in the classical sense, but not in our sense.  

We give a geometric characterization of the generic topological type:

\begin{theo}\label{Thm:main}
Let $G\in \Omega_2(d_1,d_2)$, where $d_1d_2>2$. Assume that $G$ has $$c(G)=d_1^2+d_2^2+3d_1d_2-6d_1-6d_2+7$$ cusps and $$n(G)=\frac{1}{2}\left[(d_1d_2-4)((d_1+d_2-2)^2-2)-(\gcd(d_1,d_2)-5)(d_1+d_2-2)-6\right]$$ nodes. Then $G$ is locally stable and topologically stable in $\Omega_2(d_1,d_2)$.
In particular, mappings with the generic topological type form a Zariski open subset in $\Omega_2(d_1,d_2)$.
\end{theo}

\begin{co}
Let $F,G\in \Omega_2(d_1,d_2)$, where $d_1d_2>2$. Assume that $F$ and $G$ have $$d_1^2+d_2^2+3d_1d_2-6d_1-6d_2+7$$ cusps and $$\frac{1}{2}\left[(d_1d_2-4)((d_1+d_2-2)^2-2)-(\gcd(d_1,d_2)-5)(d_1+d_2-2)-6\right]$$ nodes. Then there exist homeomorphisms $\Phi,\Psi:\C^2\to \C^2$ such that $$G=\Phi\circ F\circ \Psi.$$
\end{co}

\begin{re}
In Definition \ref{dfGenCus} we introduce the notion of a generalized cusp. In Theorem \ref{Thm:main} we may count generalized cusps instead of cusps. If there are $d_1^2+d_2^2+3d_1d_2-6d_1-6d_2+7$ of them, then in fact they all must be cusps (see Corollary \ref{cor_GCInd2}).

Similarly, in Definition \ref{df_GenNode} we introduce the notion of a generalized node. If the number of generalized nodes is as in Theorem \ref{Thm:main} then they all must be nodes (see Corollary \ref{disc2}).
\end{re}

\begin{re}
If we consider polynomial mappings $F:\C^2\to\C^n$ for $n>2$ then it is not always true that mappings with the generic topological type form an open subset in $\Omega_2(d_1,\ldots,d_n)$ (see \cite{fjm}).
\end{re}

This theorem gives positive answer to Conjecture 1  and  to Conjecture 2 from \cite{fj}. 
Moreover, since for a given mapping $F\in \Omega_2(d_1,d_2)$ we are able to compute the number of cusps and nodes effectively, this characterization allows us to find effectively a mapping $F\in \Omega_2(d_1,d_2)$ with a generic topological type using an algorithm of probabilistic type
(see the examples in Section \ref{secEx}).

\section{Preliminary Results}\label{secGC}

In this section our aim is to estimate the number of cusps and nodes of a mapping $F\in \Omega_2(d_1,d_2)$ with a reduced Jacobian. Throughout the paper we will assume that $d_1\geq d_2$. Moreover, we exclude the trivial case when $d_1d_2\leq 2$ (for such $d_1,d_2$ all mappings have $0$ cusps and $0$ nodes). For $F=(f,g)\in \Omega_2(d_1,d_2)$ we denote
$$J(F):=f_xg_y-f_yg_x,$$
$$J_{1,1}(F):=(f_{xx}g_y+f_xg_{xy}-f_{xy}g_x-f_yg_{xx})f_y-(f_{xy}g_y+f_xg_{yy}-f_{yy}g_x-f_yg_{xy})f_x,$$
and
$$J_{1,2}(F):=(f_{xx}g_y+f_xg_{xy}-f_{xy}g_x-f_yg_{xx})g_y-(f_{xy}g_y+f_xg_{yy}-f_{yy}g_x-f_yg_{xy})g_x.$$
If $F$ is a generically finite mapping, then we consider the critical set $C(F):=\{J(F)=0\}$ and the discriminant $\Delta (F):=F(C(F))$.
In the sequel the following notions play a crucial role:

\begin{defi}
Let $F\colon (\C^2,a)\to (\C^2,F(a))$ be a holomorphic mapping.  We say
that $F$ has a \emph{fold} at $a$ if $F$ is biholomorphically
equivalent to the mapping $(\C^2,0)\ni (x,y)\mapsto (x, y^2)\in
(\C^2,0)$. Moreover, we say
that $F$ has a \emph{cusp} at $a$ if $F$ is biholomorphically
equivalent to the mapping $(\C^2,0)\ni (x,y)\mapsto (x, y^3+xy)\in
(\C^2,0)$. By a \emph{node} we mean the multisingularity which consists of exactly two transversal folds: $\begin{cases} (x_1,y_1)\mapsto (x_1^2,y_1)
\\ (x_2,y_2)\mapsto (x_2,y_2^2)\end{cases}$\!.
\end{defi}

First we recall some basic facts and definitions from \cite{fjr}:

\begin{theo}[\cite{fjr}, Theorem 4.8]\label{tw1}
There is a Zariski open, dense subset $U\subset \Omega_2(d_1,d_2)$
such that for every mapping $F\in U$ the discriminant $\Delta(F)=F(C(F))$ has only cusps and nodes as singularities. 
The number of cusps is equal to
$$c(d_1,d_2)=d_1^2+d_2^2+3d_1d_2-6d_1-6d_2+7$$
and the number of nodes is equal to
$$n(d_1,d_2)=\frac{1}{2}\left[(d_1d_2-4)((d_1+d_2-2)^2-2)-(\gcd(d_1,d_2)-5)(d_1+d_2-2)-6\right].$$
\end{theo}

Let us note that for every generic mapping $F\in \Omega_2(d_1,d_2)$ we have $\deg C(F)=R:=d_1+d_2-2$ and $\deg \Delta(F)=D:=(d_1+d_2-2)d_1$.
Hence for every mapping $F\in \Omega_2(d_1,d_2)$ we have $\deg C(F)=R_F\le R$ and $\deg \Delta(F)=D_F\le D$. 

\begin{defi}\label{dfGenCus}
Let $F \colon (\C^2,a)\to (\C^2,F(a))$ be a germ of a holomorphic mapping. We say
that $F$ has a \emph{generalized cusp} at $a$ if $F_a$ is proper, the
curve $J(F)=0$ is reduced near $a$ and the discriminant of $F_a$
is not smooth at $F(a)$.
\end{defi}

\begin{re}
If $F_a$ is proper, $J(F)=0$ is reduced near $a$ and $J(F)$ is
singular at $a$ then it follows from Theorem 1.14 from \cite{jel} that
also the discriminant of $F_a$ is singular at $F(a)$ and hence
$F$ has a generalized cusp at $a$.
\end{re}

Now we introduce the index of generalized cusp:

\begin{defi}\label{dfGenCusIn}
Let $F=(f,g): (\C^2,a)\to (\C^2,F(a))$ be a germ of a holomorphic mapping.
Assume that $F$ has a generalized cusp at a point $a\in \C^2$.
Since the curve $J(F)=0$ is reduced near $a$, we have that the set
$\{\nabla f=0\}\cap \{\nabla g=0\}$ has only isolated points near
$a$. For a generic linear mapping $T\in GL(2)$, if
$F'=(f',g')=T\circ F$ then $\nabla f'$ does not vanish
identically on any branch of $\{J(F)=0\}$ near $a$. We say that
the cusp of $F$ at $a$ has an index $\mu_a:={\dim}_\C {\mathcal
O}_a/(J(F'), J_{1,1}(F'))-{\dim}_\C {\mathcal O}_a/(f'_x, f'_y)$.
\end{defi}

\begin{re}
Using the exact sequence $1.7$ from \cite{gm1} we see that
$$\mu_a={\dim}_\C {\mathcal O}_a/(J(F), J_{1,1}(F), J_{1,2}(F)).$$
Hence our index coincides with the classical local number of cusps
defined e.g. in \cite{gm1}.

Moreover, if $\mu_a(F)$ is a positive number for some germ $F_a$, then $F$ has a generalized cusp in $a$. It is easy to see that if $a$ is a simple cusp of $F$ then $\mu_a(F)=1$.
\end{re}

We have:

\begin{theo}[see {\cite[Theorem 7.6]{fjr}}]\label{tw}
Let $F=(f,g)\in \Omega_2(d_1,d_2)$. Assume that $F$ has a
generalized cusp at $a\in \C^2$. If $U_a\subset \C^2$ is a sufficiently
small ball around $a$ then $\mu_a$ is equal to the number of
simple cusps in $U_a$ of a mapping $F'$ where $F'\in \Omega_2(d_1,d_2)$ is a generic mapping, which is sufficiently
close to $F$ in the natural topology of $\Omega_2(d_1,d_2)$. 
\end{theo}

\begin{proof}
We can assume that $\nabla f$ does not vanish identically on any
branch of $\{J(F)=0\}$ near $a$. In particular we have
$\dim \mathcal{O}_a/(f_x,f_y)=\dim\mathcal{O}_a/(J(F),f_x,f_y)<\infty$.

Let $F_i=(f_i, g_i)\in\Omega_2(d_1,d_2)$ be a sequence of
generic mappings, which is convergent to $F$. Consider the
mappings $\Phi=(J(F), J_{1,1}(F))$, $\Phi_i=(J(F_i),
J_{1,1}(F_i))$, $\Psi=(\nabla f)$ and $\Psi_i=(\nabla f_i)$. Obviously
$\Phi_i\to \Phi$ and $\Psi_i\to \Psi$.

Since $a$ is a cusp of $F$ we have $\Phi(a)=0$. Moreover
$d_a(\Phi)<\infty,$ where $d_a(\Phi)$ denotes the local
topological degree of $\Phi$ at $a$. Indeed, if $J_{1,1}(F)=0$ on
some branch $B$ of the curve $J(F)=0$ then the rank of $F_{|B}$
would be zero and by Sard theorem $F$ has to contract $B$, which
is a contradiction ($F_a$ is proper). By the Rouche Theorem (see \cite{cir}, p. 86), we
have that for large $i$ the mapping $\Phi_i$ has exactly
$d_a(\Phi)$ zeroes in $U_a$ and $\Psi_i$ has exactly
$d_a(\Psi)$ zeroes in $U_a$ (counted with multiplicities, if
$\Psi(a)\not=0$ we put $d_a(\Psi)=0$). However, the mappings $F_i$
are generic, in particular all zeroes of $\Phi_i$ and $\Psi_i$ are
simple. Moreover the zeroes of $\Phi_i$ which are not cusps of $F_i$ are
zeroes of $\Psi_i$.
Hence $\mu_a=d_a(\Phi)-d_a(\Psi)$ is indeed the number of
simple cusps of $F_i$ in $U_a$.
\end{proof}

\begin{co}\label{cor_GCInd}
Let $F\in \Omega_2(d_1,d_2)$. Assume that $F$ is proper and has a reduced Jacobian. Then $F$ has only folds and generalized cusps as mono-singularities.
Moreover if $a_1,\ldots,a_m$ are generalized cusps, then $\sum^m_{i=1} \mu_{a_i}(F)\le c(d_1,d_2)=d_1^2+d_2^2+3d_1d_2-6d_1-6d_2+7$.
\end{co}

\begin{proof}
Indeed, let $p$ be a singular point of $F$. Denote by $\Delta$ the discriminant of $F$. We have two cases: 
\begin{enumerate}
\item $\Delta_p$ is a smooth germ,
\item $\Delta_p$ is not a smooth germ.
\end{enumerate}

In the first case by \cite{adv} we have that the germ $(F,p)$ is equivalent to the germ of the mapping $(x,y)\to (x^k,y)$. Since the Jacobian of $F$ is reduced we have $k=2$.

In the second case $F$ has a generalized cusp at $p$ by definition.
Now it is enough to apply Theorem \ref{tw} and Theorem \ref{tw1}.
\end{proof}

We have the following useful characterization. 

\begin{theo}\label{cusp}
Let $F\in \Omega_2(d_1,d_2)$ and assume that $F$ has generalized cusp at $a\in\C^2$. Then $F$ has a cusp at $a$ if and only if $\mu_a(F)=1$.
\end{theo}

\begin{proof}
Of course if $F$ has a cusp at $a$, then $\mu_a(F)=1$.

Conversely, assume that $\mu_a(F)=1$. We can assume that $a=(0,0)$ and $F(a)=(0,0)$. Notice that $F$ has rank $1$ at $a$. Indeed, otherwise we would have $1,x,y\not\in (J(F), J_{1,1}(F), J_{1,2}(F))$, hence $\mu_a(F)\ge 3$. Thus we can assume $F_a=(x, h(x,y))$, where $h(0,0)=0$.

Furthermore, $$\mu_a(F)=\dim_{\C}\ {\mathcal O}_a/(h_y,h_{yy})=1,$$ hence
$h_y(a)=h_{yy}(a)=0$ and  $\{h_y=0\}\pitchfork \{h_{yy}=0\}$ at $a$.  Consequently we have $h_{xy}(a)\not=0$ and $h_{yyy}\not=0$.

By the Weierstrass Preparation Theorem we have
$$h(x,y)=v(x,y)(y^3+a_1(x)y^2+a_2(x)y+a_3(x)),$$
where $v, a_i$ are holomorphic and $v(0,0)=C\not=0$, $a_i(0)=0$ for $i=1,2,3$.
By the Rouche Theorem we see that $d(F_a)=d(x,y^3+a_1(x)y^2+a_2(x)y+a_3(x))=3$
(here $d(F)$ denotes the topological degree of the mapping $F$).
Hence $F_a$ is a local analytic covering of degree $3$. By \cite{gun}, Theorem 12, p. 104, we have
\begin{equation}\label{eqn1}
y^3+b_1(x,h)y^2+b_2(x,h)y+b_3(x,h)=0.\tag{$\star$}
\end{equation}
Here $b_i$ are holomorphic and $b_i(0,0)=0$. The equality (\ref{eqn1}) can be rewritten as
\begin{equation}\label{eqn2}
(y+b_1(x,h)/3)^3+c(x,h)(y+b_1(x,h)/3)=d(x,h),\tag{$\star\star$}
\end{equation}
where $c(z,w)$, $d(z,w)$ are holomorphic and $c(0,0)=d(0,0)=0$.
Since $h(0,y)=Cy^3+\ldots$, we have from (\ref{eqn2}) that
$\frac{\partial{d}}{\partial w}(0,0)=C^{-1}\not=0$.
Hence $\frac{\partial c}{\partial z}(0,0)=C^{-1}\frac{\partial^2{h}}{\partial x\partial y}(0,0)\not=0$.
Furthermore, $\frac{\partial c}{\partial w}(0,0)=0$, because there is no term $y^2$ on the right hand side of (\ref{eqn2}).
Now we see that
$$x_1=c(x,h),\ y_1=y+b_1(x,h)/3$$
and
$$\overline{X}=c(x,y),\ \overline{Y}=d(x,y)$$
form systems of coordinates.

But $F_a^*(\overline{X},\overline{Y})=(c(x,h),d(x,h))=(x_1,y_1^3+x_1y_1)$.
\end{proof}

\begin{co}\label{cor_GCInd2}
Let $F\in \Omega_2(d_1,d_2)$. 
If $F$ has $m= c(d_1,d_2)=d_1^2+d_2^2+3d_1d_2-6d_1-6d_2+7$ generalized cusps, then all these generalized cusps are simple.
\end{co}

\begin{proof}
Indeed, if $a_1,\ldots,a_m$ are such generalized cusps then $\sum^m_{i=1} \mu_{a_i}(F) \le m$ by Corollary \ref{cor_GCInd}. Hence $\mu_{a_i}(F) =1$
for every $i\in\{1,\ldots,m\}$ and the cusps are simple.
\end{proof}

\begin{defi}\label{df_GenNode}
Let $F=(f,g):\C^2\to\C^2$ be a proper polynomial mapping. Denote by $C(F)=\{x:\ J(F)(x)=0\}$ its critical set. By a \emph{generalized node} of $F$, denoted by $(F,a,b;c)$,  we mean a non-ordered pair $a,b\in C(F)$  such that $F(a)=F(b)=c$ and the point $(a,b)$ is an isolated zero of the mapping $G\ni (x,y)\in C(F)\times C(F)\mapsto F(x)-F(y) \in \C^2$. By the \emph{multiplicity of the node} $(F,a,b;c)$ we mean the number $$\nu(a,b)=dim_\C \ {\mathcal O}_{(a,b)}/({J(F)(x),J(F)(y), f(x)-f(y),g(x)-g(y))}.$$
\end{defi}

\begin{re} \label{nody} If $(F,a,b;c)$ is a generalized node, then $a,b\in C(F)$, $F(a)=F(b)$ and $F(C(F)_a)\not=F(C(F)_b)$. Conversely, if $a,b\in C(F)$, $F(a)=F(b)$,  $F(C(F)_a)\not=F(C(F)_b)$ and we additionally assume that germs $C(F)_a$ and $C(F)_b$ are irreducible, then $(a,b)$ is a generalized node.
\end{re}

Of course every mapping $F\in \Omega_2(d_2,d_2)$ has only finitely many generalized nodes and for every generalized node $(F,a,b;c)$ the number $\nu(a,b)$ is positive and finite.

\begin{theo}\label{tw11}
Let $F=(f,g)\in \Omega_2(d_1,d_2)$. Assume that $F$ has a
generalized node at $(a,b)\in \C^2\times \C^2$. If $U_a, U_b\subset \C^2$ are sufficiently
small balls around $a$ and $b$, respectively, then $\nu(a,b)$ is equal to the number of
simple nodes in $U_a\times U_b$ of a mapping $F'$ where $F'\in \Omega_2(d_1,d_2)$ is a generic mapping, which is sufficiently
close to $F$ in the natural topology of $\Omega_2(d_1,d_2)$. 
\end{theo}

\begin{proof}
Let $F_i=(f_i, g_i)\in\Omega_2(d_1,d_2)$ be a sequence of
generic mappings, which is convergent to $F$. Consider the
mappings $\Phi=(J(F)(x), J(F)(y),f(x)-f(y),g(x)-g(y))$ and $\Phi_i(x,y)=(J(F_i)(x), J(F_i)(y),f_i(x)-f_i(y),g_i(x)-g_i(y))$. Obviously
$\Phi_i\to \Phi$.

Since $(a,b)$ is a generalized node of $F$ we have $\Phi(a,b)=0$ and 
$d_{(a,b)}(\Phi)<\infty$. We can choose neighborhoods $U_a$ and $U_b$ of $a$ and $b$, respectively, so that $\Phi$ has no other zeroes in $U_a\times U_b$. 
By the Rouche Theorem (see \cite{cir}, p. 86), we
have that for large $i$ the mapping $\Phi_i$ has exactly
$d_{(a,b)}(\Phi)$ zeroes in $U_a\times U_b$.  However, the mappings $F_i$
are generic, in particular all zeroes of $\Phi_i$ are
simple nodes. 
Hence $\nu(a,b)(F)=d_{(a,b)}(\Phi)$ is indeed the number of
simple nodes of $F_i$ in $U_a\times U_b$.
\end{proof}

\begin{co}\label{foldes}
Let $F\in \Omega_2(d_1,d_2)$ and let $(a_1,b_1),\ldots, (a_m,b_m)$ be generalized nodes of $F$.
Then $\sum^m_{i=1} \nu(a_i,b_i)\le n(d_1,d_2)$, in particular $F$ has at most $n(d_1,d_2)$ generalized nodes. 
\end{co}

\begin{co}\label{disc}
Let $F\in \Omega_2(d_1,d_2)$ be a proper mapping. Assume that Jacobian of $F$ does not have self-intersections.
Then the discriminant of $F$ has at most $n(d_1,d_2)$ self-intersections.
\end{co}

\begin{proof}
It follows from Remark \ref{nody} and Corollary \ref{foldes}.
\end{proof}

The mapping $F=(x^2,y^2)$ is an example that shows that we can not omit the assumption about self-intersections of the Jacobian of $F$. 

\begin{co}\label{disc1}
Let $F\in \Omega_2(d_1,d_2)$ be a proper mapping with a reduced Jacobian without self-intersections. 
Then the discriminant of $F$ has at most $c(d_1,d_2)$ simple (i.e. irreducible) singularities and at most $n(d_1,d_2)$ self-intersections.
\end{co}

\begin{proof}
It follows from Corollary \ref{cor_GCInd} and Corollary \ref{disc}.
\end{proof}

We also have a result analogous to Theorem \ref{cusp}:

\begin{theo}
Let $F=(f,g)\in \Omega_2(d_1,d_2)$. Assume that $F$ has generalized node $(F,a,b;c)$. Then $(F,a,b;c)$ is a node if and only if $\nu (a,b)(F)=1$.
\end{theo}

\begin{proof}
If $(F,a,b;c)$ is a node then by direct computation we see that $\nu (a,b)(F)=1$.

Now assume that  $(F,a,b;c)$ is a generalized node and $\nu (a,b)(F)=1$. Hence the mapping $H:=(J(F)(x),J(F)(y), f(x)-f(y),g(x)-g(y))$ is a local biholomorphism near $(a,b)$. In particular the germs $J_a:=\{J(F)=0\}_a$, $J_b:=\{J(F)=0\}_b$ are smooth. We can assume that $c=(0,0)$. Since the curves $J_a\times \{b\}\subset \C^2\times \{b\}$ and $\{a\}\times J_b\subset \{a\}\times \C^2$ have distinct tangents, so do their images via $H$, namely $\{0,0\}\times F(J_a)$ and $\{0,0\}\times F(J_b)$. Thus can assume that $F(J_a)=\{ x=0\}$ and $F(J_b)=\{ y=0\}$. Since the discriminant of the germ $F_a$ is smooth and it has reduced Jacobian we see that this germ is a fold. Similarly, $F_b$ is a fold. Arguing as in \cite{adv} we can assume that $F_a$ is equivalent to $(x_1^2,y_1)$ and $F_b$ is equivalent to $(x_2,y_1^2)$.
\end{proof}

\begin{co}\label{disc2}
Let $F\in \Omega_2(d_1,d_2)$. If $F$ has $m= n(d_1,d_2)$ generalized nodes, then all these generalized nodes are simple nodes and $F$ has no other discrete multi-singularities.
\end{co}

\begin{proof}
Indeed, if $(a_1,b_1)\ldots,(a_m,b_m)$ are such generalized nodes then $\sum^m_{i=1} \nu (a_i,b_i)(F) \le m$, hence $\nu (a_i,b_i)(F)=1$ for every $i\in\{1,\ldots,m\}$.
\end{proof}

\section{Main Result}
The aim of this section is to prove Theorem \ref{Thm:main}. We begin with two lemmas which allow us to exploit the assumption that a mapping has the maximal possible number of cusps and nodes.

\begin{lem}\label{lem_inkluzja_bezout}
Let $f_1,f_2,f_3,f_4$ be homogeneous polynomials in $\C[x_0,x_1,x_2]$ of degrees, respectively, $d_1,d_2,d_3,d_4$ with $d_1=d_2\leq d_3<d_4$. Assume that the ideal $(f_3,f_4)$ is contained in $(f_1,f_2)$. Let $N$ be the number of discrete points of $V(f_3,f_4)\subset\mathbb{P}^2$ that are not contained in $V(f_1,f_2)$. We have $N\leq d_3d_4-d_1d_2$. If additionally $V(f_3,f_4)$ contains a curve or there is a point contained in $V(f_3,f_4)$ with higher multiplicity than in $V(f_1,f_2)$ then the inequality is strict.
\end{lem}
\begin{proof}
First assume that $\gcd\{f_3,f_4\}=1$ (and consequently $\gcd\{f_1,f_2\}=1$). Then $V(f_3,f_4)$ and $V(f_1,f_2)$ are finite sets containing
by Bezout Theorem, respectively, $d_3d_4$ and $d_1d_2$ points, counted with multiplicities. The multiplicities in $V(f_3,f_4)$ of all points contained in $V(f_1,f_2)$ are greater than or equal to their multiplicities in $V(f_1,f_2)$. 
Hence
\begin{equation*}
\begin{aligned}
d_3d_4&=\sum_{p\in V(f_1,f_2)}\mult_p(f_3,f_4)+\sum_{p\notin V(f_1,f_2)}\mult_p(f_3,f_4)\\
&\geq \sum_{p\in V(f_1,f_2)}\mult_p(f_1,f_2)+N= d_1d_2+N.
\end{aligned}
\end{equation*}

Thus $N\leq d_3d_4-d_1d_2$ and if at some point $p\in V(f_1,f_2)$ the inequality between multiplicities is strict then $N< d_3d_4-d_1d_2$.

Now assume that $\gcd\{f_3,f_4\}=g$, where $d=\deg g$, $0<d<d_3$ and $\gcd\{f_1,f_2\}=1$. Note that for generic $a\in\C$ we have $\gcd\{f_1+af_2,g\}=1$, moreover, we may replace $f_1$ with $f_1+af_2$ without changing $N$. Obviously $(f_3,f_4):g=\left(\frac{f_3}{g},\frac{f_4}{g}\right)$ and for $p\notin V(g)$ and an ideal $I$ we have $\mult_p V(I)=\mult_p V(I:g)$.

Note that for a point $p\in V(f_1,f_2,g)$ we have the following exact sequence:
$$0\rightarrow\mathcal{O}_p/(f_1,f_2):g\xrightarrow{\cdot g}\mathcal{O}_p/(f_1,f_2)\rightarrow\mathcal{O}_p/(f_1,f_2,g)\rightarrow 0.$$
It follows that
\begin{equation*}
\begin{aligned}
\mult_p V(f_1,f_2)&=\mult_p V(f_1,f_2,g)+\mult_p V((f_1,f_2):g)\\
&\leq \mult_p V(f_1,g)+\mult_p V((f_3,f_4):g)\\
&=\mult_p V(f_1,g)+\mult_p V\left(\frac{f_3}{g},\frac{f_4}{g}\right).\end{aligned}
\end{equation*}
Summing over all $V(g)$ we obtain 
$$\sum_{p\in V(g)}\mult_p V(f_1,f_2)\leqslant d_1d+\sum_{p\in V(g)}\mult_p V\left(\frac{f_3}{g},\frac{f_4}{g}\right).$$

Moreover, by Bezout Theorem we have
$$\sum_{p\notin V(g)}\mult_pV(f_1,f_2)= d_1 d_2 -\sum_{p\in V(g)}\mult_pV(f_1,f_2).$$

Now we have
\begin{equation*}
\begin{aligned}
N&\leqslant (d_3-d)(d_4-d)-\sum_{p\notin V(g)}\mult_pV(f_1,f_2)-\sum_{p\in V(g)}\mult_p\left(\frac{f_3}{g},\frac{f_4}{g}\right)\\
&=(d_3-d)(d_4-d)-d_1d_2+\sum_{p\in V(g)}\mult_pV(f_1,f_2)-\sum_{p\in V(g)}\mult_p\left(\frac{f_3}{g},\frac{f_4}{g}\right)\\
&\leqslant (d_3-d)(d_4-d)-d_1d_2+d_1d<d_3d_4-d_1d_2-d(d_4-d_1)<d_3d_4-d_1d_2.
\end{aligned}
\end{equation*}

Finally, if $\gcd\{f_1,f_2\}=g$ and $d=\deg g>0$ then we take $g_i=f_i/g$ for $i=1,\ldots,4$. The discrete points of $V(f_3,f_4)$ that are not contained in $V(f_1,f_2)$ are also discrete points of $V(g_3,g_4)$ not contained in $V(g_1,g_2)$. Hence $N\leq(d_3-d)(d_4-d)-(d_1-d)(d_2-d)=d_3d_4-d_1d_2-d(d_3+d_4-d_1-d_2)<d_3d_4-d_1d_2$.

\end{proof}

For a mapping $F\in \Omega_2(d_1,d_2)$ let $c(F)$ denote the number of cusps of $F$ and let $n(F)$ denote the number of nodes. Of course 
$c(F)\le c(d_1,d_2)$ and $n(F)\le n(d_1,d_2)$.

\begin{lem}\label{lem_max_implikacje}
Let $F=(f_1,f_2)\in \Omega_2(d_1,d_2)$. Assume that $F$ has the maximal possible number of cusps and nodes, i.e., $c(F)=c(d_1,d_2)$ and $n(F)=n(d_1,d_2)$. Then:
\begin{enumerate}
\item $C(F)$ is smooth and has degree $d_1+d_2-2$,
\item $V(f_1)$ and $V(f_2)$ do not have a common point at infinity,
\item $\mu(F)=d_1d_2$ and $F$ is proper,
\item $C(F)$ is smooth at infinity and irreducible.
\item the mapping $F:C(F)\to\Delta(F)$ is birational,
\item all mono-singularities of $F$ are folds or cusps, all multi-singularities of $F$ are nodes,
\item if $\gcd(d_1,d_2)\neq d_2$ then $V(f_1)$ and $V(J(F))$ do not have a common point at infinity.
\end{enumerate}
\end{lem}
\begin{proof}
Let $\Sigma^{1,1}(F)$ denote the set of cusps of $F$. Recall that $\Sigma^{1,1}(F)$ is a finite subset of $V(J(F),J_{1,1}(F),J_{1,2}(F))$, moreover 
$\Sigma^{1,1}(F)\cap V(f_{1,x_1},f_{1,x_2},f_{2,x_1},f_{2,x_21})=\emptyset$. Note that by replacing $F$ with $(f_1+af_2,f_2)$, where $a\in\C$, we do not change the set of cusps, $V(J(F))$, the degree of $F$ (recall that we assume $d_1\geq d_2$), nor any of the assertions. However, we do change $f_{1,x_1}$, $f_{1,x_2}$ and $J_{1,1}(F)$ to, respectively, $f_{1,x_1}+af_{2,x_1}$, $f_{1,x_2}+af_{2,x_2}$ and $J_{1,1}(F)+aJ_{1,2}(F)$. Thus for a generic $a\in\C$ we can ensure, by replacing $F$ with $(f_1+af_2,f_2)$, that $\Sigma^{1,1}(F)\cap V(f_{1,x_1},f_{1,x_2})=\emptyset$ and that all curves in $V(J(F),J_{1,1}(F))$ are also contained in $V(J(F),J_{1,1}(F),J_{1,2}(F))$. 
Let $\overline{f_{1,x_1}}$ be the homogenization of $f_{1,x_1}$ of degree $d_1-1$, i.e., $\displaystyle \overline{f_{1,x_1}}(x_0,x_1,x_2)=x_0^{d_1-1}f_{1,x_1}\left(\frac{x_1}{x_0},\frac{x_2}{x_0}\right)$ regardless of the actual degree of $f_{1,x_1}$. Similarly, we define $\overline{f_{1,x_2}}, \overline{J(F)}, \overline{J_{1,1}(F)}, \overline{f_1}, \overline{f_2}$ as homogenizations of degrees $d_1-1$, $d_1+d_2-2$, $2d_1+d_2-3$, $d_1$ and $d_2$, respectively.

Note that the ideal $(\overline{J(F)}, \overline{J_{1,1}(F)})$ is contained in the ideal $(\overline{f_{1,x_1}},\overline{f_{1,x_2}})$, hence we may apply Lemma \ref{lem_inkluzja_bezout}. Moreover, by Theorem \ref{tw1} the maximal possible number of cusps is equal to the upper bound from Lemma \ref{lem_inkluzja_bezout}.

To prove (1) note that if $C(F)$ is non-reduced along a curve $C$ then $C$ is also a component of $V(J_{1,1}(F))$. Moreover, if $C(F)$ does not have the maximal degree then neither has $J_{1,1}(F)$, so the line at infinity is a component of $V(\overline{J(F)},\overline{J_{1,1}(F)})$. In both cases by Lemma \ref{lem_inkluzja_bezout} $F$ does not have the maximal number of cusps. Moreover, if $C(F)$ has a singular point $p$, then $F$ has a generalized cusp at $p$ which is not a cusp, which together with the assumption $c(F)=c(d_1,d_2)$ contradicts Corollary \ref{cor_GCInd}. Hence $C(F)$ is smooth.

To prove (2) assume that $V(f_1)$ and $V(f_2)$ do have a common point at infinity. After a linear change of coordinates in the domain we may assume that it is $(0:0:1)$. By Lemma \ref{lem_inkluzja_bezout} $\overline{J(F)}$ and $\overline{J_{1,1}(F)}$ do not have a common factor, thus the multiplicity of $V(\overline{J(F)}, \overline{J_{1,1}(F)})$ at $(0:0:1)$ is finite. We will show that the multiplicity of $V(\overline{J(F)}, \overline{J_{1,1}(F)})$ at $(0:0:1)$ is strictly larger than the multiplicity of $V(\overline{f_{1,x_1}},\overline{f_{1,x_2}})$ at $(0:0:1)$, so by Lemma \ref{lem_inkluzja_bezout} $F$ does not have the maximal number of cusps.

We will denote by $\widetilde{f}$ the dehomogenization of $\overline{f}$ with respect to $x_2$. By (1) $\deg f_1=d_1$, thus we may define the integer $k_1$ as the minimal $k$ such that the term $x_1^kx_2^{d_1-k}$ has a nonzero coefficient in $f_1$. By assumption $\widetilde{f_1}(0,0)=0$, so $\ord_{x_1} \widetilde{f_1}(0,x_1)=k_1>0$. Similarly, $\ord_{x_1} \widetilde{f_2}(0,x_1)=k_2>0$. Moreover, for $i=1,2$ we have $\ord_{x_1} \widetilde{f_{i,x_1}}(0,x_1)=k_i-1$ and $\ord_{x_1} \widetilde{f_{i,x_2}}(0,x_1)\in\{k_i, \infty\}$, where $\infty$ is achieved for $k_i=d_i$. It follows that $\ord_{x_1} \widetilde{J(F)}(0,x_1)\geq k_1+k_2-1\geq k_1$ and $\ord_{x_1} \widetilde{J_{1,1}(F)}(0,x_1)\geq 2k_1+k_2-2\geq k_1$.

Now, for $f$, $g$ equal, respectively, $\widetilde{f_{1,x_1}}$, $\widetilde{f_{1,x_2}}$ or $\widetilde{J(F)}$, $\widetilde{J_{1,1}(F)}$ we may use the exact sequence
$$0\rightarrow\mathcal{O}_{(0,0)}/(f,g):x_0\xrightarrow{\cdot x_0}\mathcal{O}_{(0,0)}/(f,g)\rightarrow\mathcal{O}_{(0,0)}/(f,g,x_0)\rightarrow 0$$
to obtain 
$$\mult_{(0,0)}V(f,g)=\mult_{(0,0)}V(f,g,x_0)+\mult_{(0,0)}V((f,g):(x_0)).$$
Furthermore, we have $\mult_{(0,0)}V(f,g,x_0)=\min\{\ord_{x_1}f(0,x_1),\ord_{x_1}g(0,x_1)\}$. Thus,
$$\mult_{(0,0)}V(\widetilde{f_{1,x_1}},\widetilde{f_{1,x_2}})=k_1-1+\mult_{(0,0)}V((\widetilde{f_{1,x_1}},\widetilde{f_{1,x_2}}):(x_0))$$
and 
$$\mult_{(0,0)}V(\widetilde{J(F)},\widetilde{J_{1,1}(F)})\geq k_1+\mult_{(0,0)}V((\widetilde{J(F)},\widetilde{J_{1,1}(F)}):(x_0)).$$
Since taking quotient of ideals preserves inclusion we have
$$\mult_{(0,0)}V((\widetilde{f_{1,x_1}},\widetilde{f_{1,x_2}}):(x_0))\leq \mult_{(0,0)}V((\widetilde{J(F)},\widetilde{J_{1,1}(F)}):(x_0)),$$
so we obtain the desired strict inequality of multiplicities.

To prove (3) observe that by (2) and Bezout theorem $F^{-1}(a,b)=V(f_1-a,f_2-b)$ has $d_1d_2$ points counted with multiplicities.

To prove (4) suppose that $p=(0:p_1:p_2)$ is a singular point of $V(\overline{J(F)})$. Observe that $p\in V(\overline{J(F)},\overline{J_{1,1}(F)})$. If we had $p\not\in V(\overline{f_{1,x_1}},\overline{f_{1,x_2}})$ then $p$ would be one of the at most $c(d_1,d_2)$ points that are potential cusps. Since $p$ lies at infinity it is not a cusp, so the number of cusps would not be maximal. Thus $\overline{f_{1,x_1}}(p)=\overline{f_{1,x_2}}(p)=0$. Furthermore, $$d_1\overline{f_1}(p)=0\cdot\overline{f_1}_{x_0}(p)+p_1\cdot\overline{f_1}_{x_1}(p)+p_2\cdot\overline{f_1}_{x_2}(p)=0.$$
Since $\overline{f_1}(p)=0$ we have by (2) that $\overline{f_2}(p)\neq 0$. It follows that $\overline{f_{2,x_1}}(p)\neq 0$ or $\overline{f_{2,x_2}}(p)\neq 0$. Without loss of generality we may assume that $\overline{f_{2,x_2}}(p)\neq 0$. Let $\widetilde{f}$ denote the germ at $p$ of a dehomogenization of $\overline{f}$. Now we have:
$$\mult_pV(\overline{J(F)},\overline{J_{1,1}(F)})=$$
$$\mult_pV\left(\widetilde{f_{1,x_1}}- \frac{\widetilde{f_{1,x_2}}\cdot\widetilde{f_{2,x_1}}}{\widetilde{f_{2,x_2}}},
\widetilde{f_{1,x_2}}\cdot\widetilde{J_{x_1}(F)}-\widetilde{f_{1,x_1}}\cdot \widetilde{J_{x_2}(F)}\right)=
$$
$$\mult_pV\left(\widetilde{f_{1,x_1}}- \frac{\widetilde{f_{1,x_2}}\cdot\widetilde{f_{2,x_1}}}{\widetilde{f_{2,x_2}}},
\widetilde{f_{1,x_2}}\left(\widetilde{J_{x_1}(F)}- \frac{\widetilde{J_{x_2}(F)}\cdot\widetilde{f_{2,x_1}}}{\widetilde{f_{2,x_2}}}\right)\right)=
$$
$$=\mult_pV(\widetilde{f_{1,x_1}} \cdot\widetilde{f_{2,x_2}}-\widetilde{f_{1,x_2}}\cdot\widetilde{f_{2,x_1}},\widetilde{f_{1,x_2}})+
$$ $$+
\mult_pV(\widetilde{f_{1,x_1}}\cdot\widetilde{f_{2,x_2}}- \widetilde{f_{1,x_2}}\cdot\widetilde{f_{2,x_1}},
\widetilde{J_{x_1}(F)}\cdot\widetilde{f_{2,x_2}}-
\widetilde{J_{x_2}(F)}\cdot\widetilde{f_{2,x_1}})\geq
$$
$$\geq \mult_pV(\overline{f_{1,x_1}},\overline{f_{1,x_2}})+1.$$
Hence by Lemma \ref{lem_inkluzja_bezout} $F$ does not have the maximal number of cusps.

Observe that the closure of $C(F)$ is a smooth curve in the projective plane, thus it must be irreducible.

To prove (5) observe that by (4) the curve $C(F)$ is irreducible, thus the mapping $F:C(F)\to\Delta(F)$ is either birational or generically $k:1$ for some $k>1$. Note that in the latter case the mapping $F$ does not have any nodes. This is a contradiction unless $(d_1,d_2)$ is either $(2,2)$ or $(3,1)$, i.e., unless $\deg C(F)=2$ and the maximal number of nodes is $0$. 
In the first case our assertion follows from \cite{fj}.
In the second case note that if $F:C(F)\to\Delta(F)$ is not birational, then
$\mu(F)\ge 4$ (in the fiber are two points of multiplicity $\ge 2$).
In the case $(3,1)$ we have $\mu(F)=3$, thus $F|_{C(F)}$ must be
birational.

To prove (6) note by (1) and (3) $F$ is proper and has a reduced Jacobian. Hence, we may apply Corollary \ref{cor_GCInd} and obtain that $F$ has at most $c(d_1,d_2)$ mono-singularities that are not folds. Thus if $c(F)=c(d_1,d_2)$ then all mono-singularities of $F$ must be folds or cusps.

Furthermore, by (5) all multi-singularities of $F$ are generalized nodes, hence by Corollary \ref{foldes} they are nodes.

To prove (7) we will count the nodes of $F$. Consider the set $Z=\{(x,y)\in\C^2\times\C^2\ :\ J(F)(x)=J(F)(y)=0,\ F(x)=F(y),\ x\neq y\}$. By (5) the set $Z$ is finite. Note that if $a\in\C^2$ is a node of $\Delta(F)$ then there are precisely two pairs $(x,y)\in Z$ such that $F(x)=a$. Thus the number of nodes of $F$ is bounded by $\frac{1}{2}|Z|$. To compute $|Z|$ we intersect $V(f_1(x)-f_1(y))$ with $V(f_2(x)-f_2(y))$ and obtain a surface, which decomposes into the diagonal $D=\{(x,y)\in\C^2\times\C^2\ :\ x=y\}$ and a surface $Z_1$ of degree $d_1d_2-1$. Then we intersect $Z_1$ with $V(J(F)(x))$ and obtain a curve, which decomposes into $D\cap V(J(F)(x))$ and a curve $Z_2$ of degree at most $(d_1d_2-2)(d_1+d_2-2)$. Finally, we intersect $Z_2$ with $V(J(F)(y))$ and obtain the set $Z$ and a finite set $Z_3$ contained in the diagonal. The set $Z_3$ contains the cusps of $F$ on the diagonal: $D\cap V(J(F)(x),J_{1,1}(F)(x),J_{1,2}(F)(x))$, which are contained in the intersection with multiplicity at least $2$. Moreover, the intersection $\overline{Z_2}\cap\overline{V(J(F)(y))}$ may contain some points at infinity. Indeed, if $\gcd(d_1,d_2)>1$ then it contains the points $(0:x_1:x_2:\varepsilon x_1:\varepsilon x_2)$ such that $\overline{J(F)}(0:x_1:x_2)=0$, $\varepsilon\neq 1$ and $\varepsilon^{\gcd(d_1,d_2)}=1$. Thus we obtain the bound
$$2n(F)\leq|Z|\leq (d_1d_2-2)(d_1+d_2-2)^2-2c(F)-(\gcd(d_1,d_2)-1)(d_1+d_2-2).$$

Note, that the bound coincides with the number of nodes of general mappings provided in Theorem \ref{tw1}. Suppose that $\gcd(d_1,d_2)\neq d_2$ and $V(f_1)$ and $V(J(F))$ do have a common point at infinity, we will show that $n(F)$ is not maximal. Indeed, let $(0:t_1:t_2)$ be the common point at infinity. Then for $\varepsilon^{\gcd(d_1,d_2)}\neq 1$ and $\varepsilon^{d_2}=1$ the points $(0:t_1:t_2:\varepsilon t_1:\varepsilon t_2)$ are contained in $\overline{Z_2}\cap\overline{V(J(F)(y))}$ and have not been subtracted when we were computing the bound for $n(F)$.

\end{proof}

\begin{lem}
The set $U=\{F\in \Omega_2(d_1,d_2): J(F) \text{ is reduced and of maximal degree}\}$ is open in $ \Omega_2(d_1,d_2)$.
\end{lem}

\begin{proof}
Take $F=F_{p_0}\in U$. Hence there exists a line $L=l(x,y)$ such that $\#L\cap J(F)=d_1+d_2-2$. Consider the mapping $\Psi: \Omega_2(d_1,d_2)\times \C^2\ni (p, (x,y)) \mapsto (p, J(F_p)(x,y), l(x,y))\in \Omega_2(d_1,d_2)\times \C^2$. It is a polynomial dominant mapping of geometrical degree $\deg J(F)=d_1+d_2-2$. Moreover, by the assumption the fiber over $(p_0, (0,0))$ has the maximal possible cardinality $d_1+d_2-2$. But this means that $\Psi$ has such cardinality of fibers in some open neighborhood $V\times D\subset \Omega_2(d_1,d_2)\times \C^2$ of the point $(p_0,(0,0))$. In particular the point $p_0$ belongs to $U$ with the open neighborhood $V$.
\end{proof}

\begin{lem}\label{lem_U1}
The set $U_1=\{F\in  U : c(F)=c(d_1,d_2) \text{ and } n(F)=n(d_1,d_2)\}$ is open in $\Omega_2(d_1,d_2)$. Moreover every $F\in U_1$ is locally stable.
\end{lem}

\begin{proof}
Let $F\in U_1$. Let $V$ be a small neighborhood of $F$ and  $F_t; t\in V$ be a small deformation of $F$. Cusps and nodes are locally stable so $c(F_t)\ge c(F)=c(d_1,d_2)$ and $n(F_t)\ge n(F)=n(d_1,d_2)$, consequently we have $c(F_t)=c(d_1,d_2)$ and $n(F_t)=n(d_1,d_2)$ for $t\in V_0\subset V$, where $V_0$ is a smaller neighborhood of $F.$

Moreover, by Lemma \ref{lem_max_implikacje} the mapping $F$ has only locally stable singular points.

\end{proof}

\begin{lem}\label{wazne}
Let $\Phi : U_1\times \C^2 \ni (p,(x,y))\mapsto (p,F_p(x,y))\in U_1\times \C^2$. Then:

\begin{enumerate}
\item $C(\Phi)\cap \{p\} \times \C^2 = C(F_p)$,
\item $\Delta (\Phi)\cap \{p\} \times \C^2 = \Delta (F_p)$,
\item every singular point of $\Delta (\Phi)$ is of the form $(p,q)$, where $q\in Sing(\Delta (F_p))$,
\item stratification of $X:=U_1\times \C^2$ given by
$$\{X_1,X_2,X_3\}=\{X\setminus \Delta (\Phi), \Delta (\Phi)\setminus \Sing(\Delta (\Phi)), \Sing(\Delta (\Phi))\}$$
is a Whitney stratification.
\end{enumerate}
\end{lem}

\begin{proof}
Let us note that $|d \Phi (p,(x,y))|=|dF_p(x,y)|$, which implies (1) and (2). Moreover, it is easy to observe that $\Phi|_{C(\phi)}$ is singular
only at points $(p,(x,y))$, where $(x,y)$ is a cusp of $F_p$. This implies (3), because other singular points are selfintersection points of $\Delta(\Phi)$, and consequently they come from nodes of $F_p$.

Now we prove (4). Of course the only interesting case is the pair $(X_2,X_3)$. Let $(p,(x,y))\in X_3$. Hence $(x,y)$ is either a cusp of $F_p$ or a node of $F_p$. In both cases the singularity $F: (\C^2, S)\to \C^2$ is stable (here $S$ is one point for a cusp and two points for
a node). Let $B$ be a small ball around $p$ such that $B\subset U_1$. The family $F_q,\ q\in B$ is a deformation of $F: (\C^2, S)\to \C^2$.
Since cusps and nodes are stable singularities we can assume that this deformation is trivial. This means that the pair $(X_2,X_3)$ is locally diffeomorphic  to $(\Gamma\setminus \{O\}, O)\times B$ where $\Gamma$ is either a cusp curve $\{ x^2=y^3\}$ or the cross $\{ xy=0\}$ and $O=(0,0).$
Since $(\Gamma\setminus \{O\}, O)$ is a Whitney stratification of $\Gamma$ we get that $(X_2,X_3)$ is a Whitney stratification of $\overline{X_2}.$
\end{proof}

In the sequel we need following definitions and results from \cite{dj}:

\begin{defi}\label{def_Kinfty}
Let $f:X\to\C^m$ be a polynomial dominant map where $X\subset \C^n$ is an affine algebraic set. Let  $S=\{X_\alpha\}_{\alpha\in I}$ be  a stratification of $X.$ 
We denote by $K_\infty(f|_{X_\alpha})$ the set $\{ y\in\C^m: {\rm there \ is \ a \ sequence} \ x_n\to \infty; \ x_n\in X_\alpha: ||x_n||\nu(d_{x_n}(f|_{X_\alpha}))\to 0\ {\rm and} \ f(x_n)\to y\}$ (here $\nu$ denotes the Rabier function, for details see \cite{Jelonek2005}). 
Now let $C(f,X_\alpha)$ denote the set  of points where $f|_{X_\alpha}$ is not a submersion. By $\Sing(f, S)$ we denote the set of stratified singular values of $f$, i.e.,
\begin{equation}\label{K0f} \Sing(f, S)=\bigcup_{\alpha\in I}K_0(f,X_\alpha), \end{equation}
where $K_0(f,X_\alpha)=\overline{f(C(f,X_\alpha))}.$ 
\end{defi}

By \cite[Theorem 3.3]{Jelonek2005} we have that for every $\alpha$ the set  $K_\infty(f|_{X_\alpha})$ has measure $0$ in $\C^m.$
In particular the set $K(f)$ defined below has also measure $0.$

\begin{defi}
 By $K(f)=K(f, S)$ we denote the set of stratified generalized critical values of $f$ given by
\begin{equation}\label{Kf} K(f):=\bigcup_{\substack{\alpha\in I}} (K_0(f|_{X_\alpha})\cup K_\infty(f|_{X_\alpha})). \end{equation}
\end{defi}

\begin{theo}[First isotopy lemma for non-proper maps, \cite{dj}, Th. 3.1]\label{NonProperIsotopy} Let $X\subset\C^n$ be an affine variety with an affine Whitney stratification $S$ and let $f:X\to\C^m$ be a polynomial dominant map. Let $K(f)$ be the set of stratified generalized critical values of $f$ given by  (\ref{Kf}). Then $f$ is locally trivial outside $K(f)$.
\end{theo}

\begin{re}
In fact Theorem \ref{NonProperIsotopy} works in a slightly more general setting (with the same proof):

\vspace{3mm}

{\it Let $X\subset\C^n$ be an affine variety and let $f:X\to\C^m$ be a polynomial dominant map. Let $U\subset \C^m$ be a Zarsiki open subset of $\C^m$ and let $X'=f^{-1}(U)$. Consider the mapping $f': X'\ni x \mapsto f(x)\in U$. Assume that $S$ is a Whitney stratification of $X'$. Let $K(f')$ be the set of stratified generalized critical values of $f'$ given by (\ref{Kf}). Then $f'$ is locally trivial outside $K(f')$.}
\end{re}

\begin{co}\label{NonProperIsotopySubmersion} Let $X\subset\C^n$ be an affine variety and let $f:X\to\C^m$ be a polynomial dominant map. Let $U\subset \C^m$ be a Zarsiki open subset of $\C^m$ and let $X'=f^{-1}(U)$. Assume that $S$ is a Whitney stratification of $X'$ such that for every stratum $X_\beta\in S$, the restriction $f|_{X_\beta}$ is a submersion and $K_\infty(f|_{X_\beta})=\emptyset$. Then $f|_{X'}$ is a locally trivial fibration.
\end{co}

We will also use a result from \cite{Jel5} which will allow us to simplify computing of $K_\infty(f|_{X_\alpha})$. Let $W_1,W_2$ be vector spaces of dimensions, respectively, $n$ and $m$. Let $H\subset W_1$ be a subspace such that $\dim H=n-r\geq m$. Let $A\in\mathcal{L}(W_1,W_2)$ be a linear mapping and $[a_{ij}]$ its matrix. Let $B_i=\sum b_{ij}x_j$, $i=1,\ldots,r$ be independent linear equations giving $H$ in $W_1$. Let $M$ be the $(m+r)\times n$ matrix obtained by appending the $r$ rows of $[b_{ij}]$ to the matrix $[a_{ij}]$. Let $M_I$, where $I=(i_1,\ldots,i_{m+r})$, denote $(m+r)\times(m+r)$ minor of $M$ given by columns indexed by $I$ and let $|M_I|$ be the modulus of the determinant of $M_I$. Similarly, let $M_J(j)$ denote a $(m+r-1)\times(m+r-1)$ minor of $M$ given by columns indexed by $J$ and deleting the $j$-th row. We define the function

$$g'(A,H)=\max_I\left\{\min_{J\subset I,1\leq j\leq m}\frac{|M_I|}{|M_J(j)|}\right\} ,$$

where we consider only indices with $|M_J(j)|\neq 0$, if all $|M_J(j)|$ are zero, we put $g'(A,H)=0$.

We have the following:

\begin{lem}[\cite{Jel5},Corollary 2.3]
The functions $\nu(\res_H A)$ and $g'(A,H)$ are equivalent, i.e., there are positive constants $C_1,C_2$ such that $C_1 g'(A,H)\leq\nu(\res_H A)\leq C_2 g'(A,H)$. In particular we can replace the Rabier function in Definition \ref{def_Kinfty} with $g'$.
\end{lem}

Now we can prove the main result: 

\begin{proof}[Proof of Theorem \ref{Thm:main}]
The proof in the case $d_2<d_1$ is slightly different from the proof for $d_1=d_2$. If $d_1=d_2$ then we define $U_2=U_1$ (see Lemma \ref{lem_U1}). If $d_2<d_1$ then we define $U_2=\{F=(f_1,f_2)\in U_1:\ V(f_1) \text{ and } V(J(F)) \text{ do not intersect at infinity}\}$, which is an open subset of $U_1$. If $d_1=d_2$ then by the assumption $G\in U_1=U_2$. If $d_2<d_1$ and $d_2$ does not divide $d_1$ then by the assumption and Lemma \ref{lem_max_implikacje} we have $G\in U_2$. Finally, if $d_2<d_1$ and $d_2$ does divide $d_1$ then $G\in U_1$, however, we will first show that $G$ is topologically stable for $G\in U_2$ and then infer that topological stability holds for $G\in U_1$.

Let us consider the space $X=U_2\times \C^2$ and its Whitney stratification $S=\{X_1,X_2,X_3\}$ as in Lemma \ref{wazne}. Let $\pi : X\ni (F,y)\mapsto F \in U_2$ be the projection. We show that $K(\pi, S)=\emptyset$ and particularly $\pi$ has a trivialization which preserves strata. Of course the Rabier function $\nu(\pi)$ on $X_1$ is everywhere equal to one and consequently 
$K(\pi, X_1)=\emptyset$. On $X_3$ the mapping $\pi$ induces a topological covering of degree $c(d_1,d_2)+n(d_1,d_2)$, hence also $K(\pi, X_3)=\emptyset$.
It remains to verify that $K(\pi, X_2)=K_\infty(\pi, X_2)=\emptyset$.

Recall that $\Phi : U_2\times \C^2 \ni (F,x)\mapsto (F,F(x))\in U_2\times \C^2$. Take $(F,y)\in X_2$. We will compute $g'(A,H)$ for $A=A(F)=d\pi(F,y)$ and $H=H(F)=T_{(F,y)}X_2$. Note that there is a unique $x=(x_1,x_2)\in C(F)$ such that $y=F(x)$. Moreover, $\Phi|C(\Phi)$ is an immersion at $(F,x)$, so $H=d\Phi(T_{(F,x)}C(\Phi))$. Furthermore, $\dim(H)=\dim(\im(d\Phi))=\dim(\Omega_2(d_1,d_2))+1$, so those spaces must be equal.

We have
$$d\Phi(F,x)=\left[\begin{matrix}I&0&0&0\\0&I&0&0\\x_1^ix_2^j,0\leq i+j\leq d_1&0&f_{1,x_1}(x)&f_{1,x_2}(x)
\\0&x_1^ix_2^j,0\leq i+j\leq d_2&f_{2,x_1}(x)&f_{2,x_2}(x)\end{matrix}\right].$$

The image of $d\Phi(F,x)$ is generated by the columns of the matrix above. Obviously the first $\dim(\Omega_2(d_1,d_2))$ columns are linearly independent and, since $F$ has a fold at $x$, the last two are proportional. Let us consider the case $\max\{|f_{1,x_1}(x)|,|f_{2,x_1}(x)|\}\geq\max\{|f_{1,x_2}(x)|,|f_{2,x_2}(x)|\}$, the result in the other case is analogous. Then $H$ is given by the equation corresponding to the vector $[-f_{2,x_1}(x)x_1^ix_2^j,0\leq i+j\leq d_1, f_{1,x_1}(x)x_1^ix_2^j,0\leq i+j\leq d_2, f_{2,x_1}(x),-f_{1,x_1}(x)]$, for the last column we use the equality $(f_{1,x_1}f_{2,x_2}-f_{2,x_1}f_{1,x_2})(x)=0$.

Thus we have
$$M=\left[\begin{matrix}I&0&0&0\\0&I&0&0\\
-f_{2,x_1}(x)x_1^ix_2^j,0\leq i+j\leq d_1& f_{1,x_1}(x)x_1^ix_2^j,0\leq i+j\leq d_2& f_{2,x_1}(x)&-f_{1,x_1}(x)
\end{matrix}\right].$$

There are only two possible nonzero values of $M_I$, which are obtained by excluding one of the two last columns. So we have $|M_I|=|f_{1,x_1}(x)|$ or $|M_I|=|f_{2,x_1}(x)|$. We obtain $M_J(k)$ by excluding from $M_I$ the $k$-th row and either the $k$-th column or the other of the last two columns.
We obtain, respectively, $|M_J(k)|=|M_I|$ or $|M_J(k)|$ equal to one of the first $\dim(\Omega_2(d_1,d_2))$ entries in the last row. Thus
$$g'(A,H)=\min\left\{1,\frac{\max\{|f_{1,x_1}(x)|,|f_{2,x_1}(x)|\}}
{\max\{|f_{2,x_1}(x)x_1^{d_1}|, |f_{2,x_1}(x)x_2^{d_1}|,|f_{1,x_1}(x)x_1^{d_2}|, |f_{1,x_1}(x)x_2^{d_2}|\}}\right\}.$$

In particular we have $g'(A,H)\geq\min\{1,|x_1|^{-d_1},|x_2|^{-d_1}\}\geq\min\{1,\|x\|^{-d_1}\}$.

Now suppose that $F\in K_\infty(\pi, X_2)$. There is a sequence $(F_t,y_t)$ such that $F_t\rightarrow F$ and $y_t\rightarrow \infty$ and $\|(F_t,y_t)\|g'(A(F_t),H(F_t))\rightarrow 0$. Take $x_t=(x_{1,t},x_{2,t})\in C(F_t)$ such that $y_t=F_t(x_t)$. By taking a subsequence we may assume that $x_t$ is convergent to a point $x_0\in\mathbb{P}^2$. Moreover, $x_0$ lies on the line at infinity, for otherwise $F_t(x_t)$ would converge to $F(x_0)$.

Let $\tilde{f}_1$ be the homogeneous part of $f_1$ of degree $d_1$. Note that either $\tilde{f}_1(x_0)=0$ or there is a neighborhood $U_{x_0}$ of $x_0$ (in $\mathbb{P}^2$ treated as a manifold), a neighborhood $U_{F}$ of $F$ and $\varepsilon>0$ such that $\|f_{1,t}(x)\|\geq\varepsilon\|x\|^{d_1}$ for all $x\in U_{x_0}\cap\C^2$ and all $F_t=(f_{1,t},f_{2,t})$. Analogously for $f_2$ and $d_2$.

If $d_1=d_2$ then from Lemma \ref{lem_max_implikacje}(3) we obtain that $\tilde{f}_1(x_0)$ and $\tilde{f}_2(x_0)$ can not be simultaneously zero. Thus $\|(F_t,y_t)\|\cdot g'(A(F_t),H(F_t))\geq\varepsilon\|x\|^{d_1}\cdot \|x\|^{-d_1}=\varepsilon$, a contradiction.

On the other hand, if $d_1>d_2$ then since $F\in U_2$ and $x_0\in\overline{V(J(F))}$ we have $x_0\notin\overline{V(f_1)}$, i.e., $\tilde{f}_1(x_0)\neq 0$, which leads to the same contradiction as above.

Thus we have established that $K(\pi, S)=\emptyset$. Now by Corollary \ref{NonProperIsotopySubmersion} we see that there is a small neighborhood $V\subset U_2$ of the point $G$ and a continuous family of homeomorphisms $\Phi_q, q\in V$ such that $\Phi_q: \C^2\to\C^2$ and $\Phi_q(\Delta(F_q))=\Delta(G)$. In particular the family 
$\Phi_q\circ F_q, q\in V$ has a constant discriminant and constant topological degree. Now we can argue as in the proof of Theorem 4.3 in \cite{jel} to show that the family $\Phi_q\circ F_q, q\in V$ is trivial. Hence also the family $F_q, q\in V$ is trivial.

It remains only to show that for $d_2<d_1$ and $d_2$ dividing $d_1$ topological stability of all $G\in U_2$ implies topological stability for $G\in U_1$. Let $F\in U_1\setminus U_2$. We may assume that the common point at infinity of $V(f_1)$ and $V(J(F))$ is $(0:0:1)$, i.e., that $x_1$ divides $\tilde{f_1}$ and $J(\tilde{F})$. Let $\tilde{f_1}=x_1h$, we have $J(\tilde{F})=h\tilde{f}_{2,x_2}+x_1(h_{x_1}\tilde{f}_{2,x_2}- h_{x_2}\tilde{f}_{2,x_1})$, so $x_1$ divides $h\tilde{f}_{2,x_2}$. Thus either $x_1^2$ divides $\tilde{f_1}$ or $x_1$ divides both $\tilde{f_1}$ and $\tilde{f_2}$. The latter is excluded by Lemma \ref{lem_max_implikacje}(2) and we obtain that $U_2=\{F=(f_1,f_2)\in U_1:\ \tilde{f_1} \text{ does not have a multiple root}\}$. Let $F_a=\left(f_1+af_2^{d_1/d_2},f_2\right)$, for $a\in\C$. Obviously $F_a$ is topologically equivalent to $F$, we will show that for a generic choice of $a$ we have $F_a\in U_2$, which will conclude the proof. Notice that a multiple root of $\tilde{f}_1+a\tilde{f}_2^{d_1/d_2}$ must also be a root of $\tilde{f}_{1,x_1}+\frac{ad_1}{d_2}\tilde{f}_{2,x_1}f_2^{d_1/d_2-1}$ and hence also of $\tilde{f}_{1,x_1}\tilde{f}_2-\frac{d_1}{d_2}\tilde{f}_1\tilde{f}_{2,x_1}$. Thus there are finitely many points at infinity depending only on $F$ at which $\tilde{f}_1+a\tilde{f}_2^{d_1/d_2}$ may have a multiple root. Thus for almost all $a\in\C$ we have $F_a\in U_2$.
\end{proof}

\section{Examples}\label{secEx}

In this section we give examples of mappings with a generic topological type. In all cases we choose a mapping $G_0\in\Omega_2(d_1,d_2)$ and use Algorithm 1 (see Appendix) to check that $G_0$ has the maximal possible number of generalized cusps. From Corollary \ref{cor_GCInd2} we obtain that all those generalized cusps are simple cusps. Then we use Algorithm 2 to check that $G_0$ has the maximal possible number of generalized nodes. By Corollary \ref{disc1} the generalized nodes are in fact simple nodes. Finally, from Theorem \ref{Thm:main} we obtain that $G_0$ has a generic topological type, i.e., there is a Zariski open, dense subset $U\in\Omega_2(d_1,d_2)$ such that for all $G\in U$ there exist homeomorphisms $\Phi,\Psi:\C^2\to \C^2$ such that $$G=\Phi\circ G_0\circ \Psi.$$

The following examples were verified with Magma Computational Algebra System \cite{mag}. They represent all the cases when the mappings with generic topological type have less than $500$ nodes. The reader may verify the examples using the calculator available on the web page {\bf http://magma.maths.usyd.edu.au/calc/}, however for mappings with more than $100$ nodes it may be necessary to conduct the calculations over a large finite field in order to accommodate the $120$ seconds calculating time restriction.

\begin{ex}
Consider the case $d_2=1$. Obviously $(x,y)$ and $(x^2,y)$ have generic topological type for $d_1=1,2$, respectively. A mapping with generic topological type has $(d_1-1)(d_1-2)$ cusps and $(d_1-1)(d_1-2)(d_1-3)/2$ nodes. For $2<d_1\leq 12$ we verified that the mapping
$$(x^{d_1}+xy^{d_1-1}+x^2y^{d_1-2}+x^{d_1-1}+x,y)$$
satisfies those requirements.
\end{ex}

\begin{ex}
Consider the case $d_2=2$. The space $\Omega_2(2,2)$ was thoroughly investigated in \cite{fj}, in particular it was shown that $(x^2+y,x+y^2)$ has generic topological type. A mapping with generic topological type has $d_1^2-1$ cusps and $[2d_1^3-4d_1^2+d_1+2-\gcd(d_1,2)d_1]/2$ nodes. For $2<d_1\leq 8$ we verified that the mapping
$$(x^{d_1}+xy^{d_1-1}+x^2y^{d_1-2}+x^{d_1-1}+x,x^2+y^2+y)$$
satisfies those requirements.
\end{ex}

\begin{ex}
Consider the case $d_2=3$. A mapping with generic topological type has $d_1^2+3d_1-2$ cusps and $[3d_1^3+2d_1^2-6d_1+3-\gcd(d_1,3)(d_1+1)]/2$ nodes. For $3\leq d_1\leq 6$ we verified that the mapping
$$(x^{d_1}+2x^{d_1-1}y+xy^{d_1-1}+x^2y^{d_1-2}+x^{d_1-1}+x,x^3+2y^3+y)$$
satisfies those requirements.
\end{ex}

\begin{ex}
Consider the case $d_2=4$. We verified that the mapping
$$(x^{d_1}+2x^{d_1-1}y+xy^{d_1-1}+x^2y^{d_1-2}+x^{d_1-1}+x,x^4+2y^4+y)$$
has $39$ cusps and $204$ nodes for $d_1=4$ and $54$ cusps and $387$ nodes for $d_1=5$. Thus in those cases it has generic topological type.
\end{ex}

\section{Appendix}

Algorithm 1 (verifies the number of cusps of $(f,g)\in\Omega_2(d_1,d_2)$)
\begin{verbatim}
Q:=RationalField(); R<x,y>:=PolynomialRing(Q,2,``lex");
F:=f(x,y); G:=g(x,y);
J:=Derivative(F,x)*Derivative(G,y)-Derivative(F,y)*Derivative(G,x);
J_1:=Derivative(F,x)*Derivative(J,y)-Derivative(F,y)*Derivative(J,x);
J_2:=Derivative(J,x)*Derivative(G,y)-Derivative(J,y)*Derivative(G,x);
I:=ideal<R | J,J_1,J_2>;
H:=Radical(EliminationIdeal(I,{x}));
Degree(Basis(H)[1]);
\end{verbatim}

The algorithm above returns a lower bound on the number of generalized cusps of $(f,g)$. By Corollary \ref{cor_GCInd2} if the result is $c(d_1,d_2)$ then it is the number of cusps of $(f,g)$. Note that if $(f,g)$ has multiple cusps with the same $x$ coordinate then the algorithm will return a number smaller than $c(d_1,d_2)$. Thus the verification may be false negative. One could remedy this by performing a generic coordinate change before computing the elimination ideal.

\vspace{3mm}

Algorithm 2 (verifies the number of nodes of $(f,g)\in\Omega_2(d_1,d_2)$)
\begin{verbatim}
Q:=RationalField(); R<u,t,p,q,x,y>:=PolynomialRing(Q,6,``lex");
F:=f(x,y); G:=g(x,y);
J:=Derivative(F,x)*Derivative(G,y)-Derivative(F,y)*Derivative(G,x);
F_1:=f(p,q); G_1:=g(p,q);
J_1:=Derivative(F_1,p)*Derivative(G_1,q)-Derivative(G_1,p)*Derivative(F_1,q);
L:=(x-p)*t-1;
I:=ideal<R | J,J_1,F-F_1,G-G_1,F-u,L>;
H:=Radical(EliminationIdeal(I,{u}));
Degree(Basis(H)[1]);
\end{verbatim}

The algorithm above returns a lower bound on the number of generalized nodes of $(f,g)$. By Corollary \ref{disc2} if the result is $n(d_1,d_2)$ then it is the number of nodes of $(f,g)$. Note that if $(f,g)$ has a node $(a,b;c)$ where $a$ and $b$ have the same $x$ coordinate then the algorithm will return a number smaller than $n(d_1,d_2)$. The result will be false also if there are nodes $(a_1,b_1;c_1)$ and $(a_1,b_1;c_1)$ where $c_1$ and $c_2$ have the same first coordinate. One could remedy this by replacing x-p in the definition of L with a generic combination of x-p and y-q and replacing F-u in the definition of I with $a$*F+$b$*G-u for a generic choice of $a$ and $b$.

\end{document}